\documentclass[smallextended,numbook,runningheads]{svjour3}     
\smartqed  

\usepackage{amsmath,amssymb,amsxtra,comment,graphicx,psfrag}
\usepackage{bm,mathrsfs}
\usepackage{mathtools}
\usepackage{graphicx}
\usepackage{epstopdf} 
\usepackage{bbding}
\usepackage{mathptmx}
\usepackage{algorithm}
\usepackage{algorithmic}
\usepackage{color}
\usepackage{cases} 
\usepackage{graphicx,ifthen,latexsym,mathrsfs,multicol}
\DeclareMathOperator\Real {Re}
\def\Re {{\mathbb R}}
\def\N{{\mathbb N}}
\begin{document}

\title{BDF$6$ SAV schemes for  time-fractional Allen-Cahn dissipative systems
}


\author{Fan Yu         \and
        Minghua Chen
}


\institute{F. Yu  \and  M. Chen (\Envelope)  \at
              School of Mathematics and Statistics, Gansu Key Laboratory of Applied Mathematics and Complex Systems, Lanzhou University, Lanzhou 730000, P.R. China\\
email: yuf20@lzu.edu.cn; chenmh@lzu.edu.cn
}


\maketitle

\begin{abstract}
Recently, the error analysis of BDF$k$ $(1\leqslant k\leqslant5)$ SAV (scalar auxiliary variable) schemes are given in \cite{Huangg:20} for the classical Allen-Cahn equation.
However, it remains unavailable for BDF$6$ SAV schemes.  In this paper, we construct and analyze BDF$6$ SAV schemes for the time-fractional dissipative systems. We carry out a rigorous error analysis for the time-fractional Allen-Cahn equation, which also fills up a gap for the classical case.
Finally, numerical experiment is shown to illustrate the effectiveness of the presented methods.
As far as we know, this is the
first  SAV schemes for the time-fractional dissipative systems.
\end{abstract}

\section{Introduction}
The scalar auxiliary variable (SAV) approach was first proposed in \cite{Shen:18,Shen:19}, which is a powerful approach
to construct efficient time discretization schemes for gradient flows and to deal with the nonlinear terms in the dissipative systems. In recent years, the approach has attracted more and more attention and has been applied to various problems. Recently, the SAV approach coupled with extrapolated and linearized Runge-Kutta methods was considered for the Allen-Cahn and Cahn-Hilliard equations in \cite{Akrivis:19}.

Note that the unconditional energy stability can only be
established for the A-stable one- and two-step BDF methods for the original SAV approach in \cite{Shen:18,Shen:19}. However, it is well known that the BDF$k$ $(k\geqslant3)$ methods are not A-stable.
It is very wonderful that the error analysis is carried out for general dissipative systems in \cite{Huangg:20}, where the powerful Nevanlinna-Odeh multipliers for BDF$k$ $(1\leqslant k\leqslant5)$  play a key role. In contrast, it has been proved that the Nevanlinna-Odeh multipliers for the BDF6 method do not exist in \cite{ACYZ:20}. Fortunately, a class of six-step simple multipliers are proposed in \cite{ACYZ:20}, which makes the energy technique applicable to the error analysis of BDF$6$ SAV schemes.

The conventional Allen-Cahn equation \cite{Allen:79} was originally
developed as models for some material science applications. It has been
widely used in fluid dynamics to describe moving interfaces through a phase-field
approach \cite{Anderson:98}. In recent years, there are many researches on the time-fractional Allen-Cahn equation, where the first order time derivative is replaced by a Caputo fractional derivative with order $\alpha\in(0,1)$. In \cite{Du:20}, the Caputo fractional derivative is discretized by backward Euler method and the convergence rate $\mathcal{O}(\tau^\alpha)$ is proved for the time-fractional Allen-Cahn equation.
The authors of \cite{Tang:19} adopt $L1$ schemes  and prove the energy stability for the time-fractional Allen-Cahn equation.

In comparison with the error analysis of BDF$k$ $(1\leqslant k\leqslant5)$ SAV schemes for classical Allen-Cahn equation in \cite{Huangg:20}, the error analysis of BDF$6$ SAV schemes remains unavailable, which is the main motivation of the present work. In this paper, we apply the six-step simple multiplier in \cite{ACYZ:20} and the SAV approach in \cite{Huang:20,Huangg:20} to construct the explicit-implicit BDF$6$ SAV schemes
for time-fractional  dissipative systems. We show that the proposed BDF$6$ SAV schemes are unconditional energy stable. The main purpose of the present work is to carry out a rigorous error
analysis of BDF$6$ SAV schemes  for the time-fractional Allen-Cahn equation. To the best of our knowledge, this is the first proof for the error analysis of SAV schemes for the time-fractional Allen-Cahn equation.

An outline of the paper is organized as follows. In the next section, we construct the BDF$6$ SAV schemes
for the time-fractional dissipative systems in a unified form and prove the proposed schemes are unconditionally energy stable. In Section 3, we recall and prove some useful lemmas for the BDF$6$ method that are needed for the error analysis.  In section 4,  we present the detailed proof for the error analysis of BDF$6$ SAV schemes for the time-fractional Allen-Cahn equation. We provide numerical experiment to demonstrate the theoretical results in the last section.

\section{BDF$6$ SAV schemes for  time-fractional dissipative systems}\label{sec:1}
We use the following notations throughout the paper. Let $\Omega\in\mathbb{R}^d(d=1,2,3)$ be a bounded domain
with sufficiently smooth boundary. Let $\| \cdot \|$  denote the norm on $L^2(\Omega)$ induced by the inner product $(\cdot , \cdot )$ and $\| \cdot \|_{H^s}$  denote the norm on the usual Sobolev spaces $H^s(\Omega)$.  To simplify the notation, we denote $u(x,t)$ by $u(t)$ and use $C$
to denote a constant which is independent on the step size $\tau$.


Let $T >0$ and consider the following time-fractional dissipative systems \cite{Tang:19}
\begin{equation} \label{2.1}
\partial_t^{\alpha}\left(u-u_0\right)+\mathcal{A}u+f(u)=0,~~0<t<T,
\end{equation}
\begin{equation} \label{2.2}
\frac{d\tilde{E}(u)}{dt}=-\mathcal{K}(u),
\end{equation}
where $\mathcal{A}$  is a positive definite, selfadjoint, linear operator on $L^2(\Omega)$ and $f(u)$ is a nonlinear operator, with the initial condition $u(0)=u_0$ and the homogeneous Dirichlet boundary condition. The model \eqref{2.1} satisfies a dissipative energy law \eqref{2.2}, where $\tilde{E}(u)>-C_0$ for all $u$ is an energy functional, $\mathcal{K}(u)>0$ for all $u\neq0$.
Here the operator $\partial_t^\alpha$, with $\alpha\in(0,1)$, denotes the left-sided
Riemann-Liouville fractional derivative in time \cite{Podlubny:99}
\begin{equation*}
\partial_t^\alpha u(t)=\frac{1}{\Gamma(1-\alpha)} \frac{d}{dt} \int_{0}^t (t-s)^{-\alpha}  u(s)ds.
\end{equation*}
Under the initial condition $u(0)=u_0$, the Riemann-Liouville  time fractional  derivative $\partial_t^\alpha \left(u(t)-u_0\right)$ in the model \eqref{2.1} is identical with the usual Caputo time fractional derivative.
\subsection{BDF$6$ SAV schemes for general time-fractional dissipative systems}
Let $N\in \N,$ $\tau:=T/N$ be the time step, and $t_n :=n \tau,
n=0,\dotsc ,N,$ be a uniform partition of the interval $[0,T].$ We first introduce Lubich's convolution quadrature \cite{Lubich:86},
i.e., the Riemann-Liouville fractional derivative $\partial_t^{\alpha}\varphi(t_n)$ can be approximated by
\begin{equation*}
 \bar{\partial}_{\tau}^{\alpha}\varphi^n:=\frac{1}{{\tau}^{\alpha}}\sum_{j=0}^{n}g_j\varphi^{n-j},
\end{equation*}
with $\varphi^{n}=\varphi(t_n)$, where the the coefficients $\{g_j\}_{j=0}^{\infty}$ are determined by the (BDF$6$ method) generating power series  $g(\zeta)$,
\begin{equation*}
g(\zeta)=\left(\sum_{j=1}^{6}\frac{1}{j}(1-\zeta)^{j}\right)^\alpha=\sum_{j=0}^{\infty}g_j{\zeta}^{j}.
\end{equation*}

We introduce the following BDF$6$ SAV schemes inspired by the six-step simple multiplier in \cite{ACYZ:20} and the SAV approach introduced in \cite{Huang:20,Huangg:20}. The key for the SAV approach is to introduce a scalar auxiliary variable (SAV).  Setting $r(t)=E(u)(t) := \tilde{E}(u)(t)+C_0>0$, we rewrite the energy law \eqref{2.2} as the following expanded system
\begin{equation*}
\frac{dE(u)}{dt}=-\frac{r(t)}{E(u)(t)}\mathcal{K}(u).
\end{equation*}
We construct the BDF$6$ SAV schemes based on the implicit-explicit BDF$6$ formulae in the following unified form:

Given $u^0=\bar{u}^0$, $r^0=E(u^0)$, we compute $\bar{u}^n,r^n,\xi^n$ and $u^n$ consecutively by
\begin{equation} \label{2.5a}
\tau^{-\alpha}\left(g_0\bar{u}^n+\sum_{j=1}^{n}g_ju^{n-j}\right)+\mathcal{A}\bar{u}^n+f\left[B_6(\bar{u}^{n-1})\right]=\tau^{-\alpha}\sum_{j=0}^{n}g_ju^0,
\end{equation}
\begin{equation} \label{2.5b}
\frac{1}{\tau}\left(r^n-r^{n-1}\right)=-\frac{r^n}{E(\bar{u}^n)}\mathcal{K}(\bar{u}^n),
\end{equation}
\begin{equation} \label{2.5c}
\xi^n=\frac{r^n}{E(\bar{u}^n)},
\end{equation}
\begin{equation} \label{2.5d}
u^n=\eta^n\bar{u}^n \quad\eta^n=1-\left(1-\xi^n\right)^{8},
\end{equation}
where $B_6(\bar{u}^{n-1})=6\bar{u}^{n-1}-15\bar{u}^{n-2}+20\bar{u}^{n-3}-15\bar{u}^{n-4}+6\bar{u}^{n-5}-\bar{u}^{n-6}$.

\subsection{BDF$6$ SAV schemes for time-fractional Allen-Cahn  dissipative systems}
Let us  consider the following time-fractional Allen-Cahn  equation \cite{Tang:19},
\begin{equation} \label{3.1}
\partial_t^{\alpha}\left(u-u_0\right)-\Delta u+f(u)=0,
\end{equation}
which  is a special case of \eqref{2.1} with $\mathcal{A}=-\Delta$, and satisfies the dissipation law \eqref{2.2}, with the initial condition $u(0)=u_0$ and the homogeneous Dirichlet boundary condition.
An important feature of the Allen-Cahn equation is that
it can be viewed as the gradient flow in $L^2$ of the Lyapunov energy
functional $E(u)=\frac{1}{2}(\mathcal{L}u,u)+(F(u),1)$, where $(\mathcal{L}u,u)=(\nabla u,\nabla u)$, the Ginzburg-Landau double-well potential $F(u)=\frac{1}{4}(u^2-1)^2$ and $f(u)=F'(u)=u^3-u$.
Without loss of generality, we shall assume that the potential function $G(u)$ satisfies the following condition:
there exists a finite constant $L$ such that
\begin{equation} \label{3.3}
\int_\Omega F(v)dx\geqslant \b{C}>0,~~\forall v,\quad \max_{u\in\mathbb{R}}\left|f'(u)\right|\leqslant L.
\end{equation}
We recursively define a sequence of approximations $u^n$ to the nodal values $u(t_n)$ by the BDF$6$. Correspondingly, the standard implicit-explicit BDF$6$ scheme for solving \eqref{3.1} seeks approximations $u^n, n=1,...,N$ to the analytic  solution $u(t_n)$  by  \cite{Lubich:86}
\begin{equation*}
\bar{\partial}_{\tau}^{\alpha}(u^n-u^0)-\Delta u^n+f\left[B_6(\bar{u}^{n-1})\right]=0,\quad u^0=u_0.
\end{equation*}
Taking $w^n:=u^n-u^0$ with $w^0=0$, we can rewrite the above equation as
\begin{equation*}
\bar{\partial}_{\tau}^{\alpha}w^n-\Delta w^n+f\left[B_6(\bar{u}^{n-1})\right]=\Delta u^0.
\end{equation*}
Correspondingly, taking $w(t):=u(t)-u_0$ with $w(t)=0$, we can rewrite \eqref{3.1} as
\begin{equation}\label{3.4}
\partial_t^{\alpha}w-\Delta w+f(u)=\Delta u_0,~~0<t<T.
\end{equation}
For \eqref{3.4}, the BDF$6$ SAV version of \eqref{2.5a} reads:
\begin{equation} \label{3.5}
\tau^{-\alpha}\left(g_0\bar{w}^n+\sum_{j=1}^n g_jw^{n-j}\right)-\Delta\bar{w}^n+f\left[B_6(\bar{u}^{n-1})\right]=\Delta u_0.
\end{equation}

\subsection{A stability result}
The following stability results of the above BDF$6$ SAV schemes are valid for general time-fractional dissipative systems.
\begin{theorem} \label{theorem1}	
Given $r^{n-1}\geqslant0$, we have $r^n\geqslant0,\xi^n\geqslant0$, and the scheme \eqref{2.5a}, \eqref{2.5b}, \eqref{2.5c}, \eqref{2.5d} for BDF$6$ is unconditionally energy stable in the sense that
\begin{equation} \label{2.12}
r^n-r^{n-1}=-\tau\xi^n\mathcal{K}(\bar{u}^n)\leqslant0.
\end{equation}
Furthermore, if $E(u)=\frac{1}{2}\left(\mathcal{L}u,u\right)+E_1(u)$ with $\mathcal{L}$ positive and $E_1(u)$ bounded from below, there exists
$M>0$ such that
\begin{equation} \label{2.13}
(\mathcal{L}u^n,u^n)\leqslant M^2,\forall n.
\end{equation}
\end{theorem}
\begin{proof}
For bringing convenience to the reader, we reprove this theorem; the proof is almost the same as that of \cite{Huangg:20}.
Given $r^{n-1}\geqslant0$ and since $E(\bar{u}^n)>0$, it follows from \eqref{2.5b} that
\begin{equation*}
r^n=\frac{r^{n-1}}{1+\tau\frac{\mathcal{K}(\bar{u}^n)}{E(\bar{u}^n)}}\geqslant0.
\end{equation*}
Then we derive from \eqref{2.5c} that $\xi^n\geqslant0$ and obtain \eqref{2.12}.
Thus, \eqref{2.12} implies $r^n\leqslant r^0,\forall n$.

Without loss of generality, we can assume $E_1(u)>1$ for all $u$. It follows from \eqref{2.5c} that
\begin{equation} \label{2.14}
|\xi^n|=\frac{r^n}{E(\bar{u}^n)}\leqslant\frac{2r^0}{(\mathcal{L}\bar{u}^n,\bar{u}^n)+2}.
\end{equation}
From  $\eta^n=1-(1-\xi^n)^{8}$ in \eqref{2.5d}, we have $\eta^n=\xi^nP_7(\xi^n)$ with $P_7$ being a polynomial of degree $7$.
Then, we derive from \eqref{2.14} that there exists $M>0$ such that
\begin{equation*}
|\eta^n|=|\xi^nP_{7}(\xi^n)|\leqslant\frac{M}{(\mathcal{L}\bar{u}^n,\bar{u}^n)+2}.
\end{equation*}
According to $u^n=\eta^n\bar{u}^n$ in \eqref{2.5d}, it implies
\begin{equation*}
\left(\mathcal{L}u^n,u^n\right)=(\eta^n)^2\left(\mathcal{L}\bar{u}^n,\bar{u}^n\right)\leqslant\left(\frac{M}{(\mathcal{L}\bar{u}^n,\bar{u}^n)+2}\right)^2(\mathcal{L}\bar{u}^n,\bar{u}^n)\leqslant M^2.
\end{equation*}
The proof is completed.
\end{proof}
\section{A few technical lemmas}
Before we proceed, for the reader's convenience, we recall the notion of the generating function of
an $n\times n$ Toeplitz  matrix $T_n$ as well as an auxiliary result, the Grenander--Szeg\"o theorem.

\begin{definition}\cite[p.\,27]{Quarteroni:07}\label{definition2.7}
A matrix $A \in \Re^{n\times n}$ is said to be positive definite in $\Re^{n}$ if $(Ax,x)>0$, $\forall x \in \Re^{n}$, $x\neq 0$.
\end{definition}

\begin{lemma}\cite[p.\,28]{Quarteroni:07}\label{lemma2.4}
A real matrix $A$ of order $n$ is positive definite  if and only if  its symmetric part $H=\frac{A+A^T}{2}$ is positive definite.
Let $H \in \Re^{n\times n}$ be symmetric. Then $H$ is positive definite if and only if the eigenvalues of $H$ are positive.
\end{lemma}

\begin{definition}\cite[p.\,13]{Chan:07} (the generating function of a Toeplitz matrix) \label{De:gen-funct}
Consider the $n \times n$ Toeplitz  matrix  $T_n=(t_{ij})\in \Re^{n,n}$ with diagonal entries $t_0,$ subdiagonal entries
$t_1,$ superdiagonal entries $t_{-1},$ and so on, and $(n,1)$ and $(1,n)$ entries
$t_{n-1}$ and   $t_{1-n}$, respectively, i.e., the entries $t_{ij}=t_{i-j}, i,j=1,\dotsc,n,$ are constant along the diagonals of $T_n.$
 Let   $t_{-n+1},\dotsc, t_{n-1}$ be the Fourier coefficients of the trigonometric polynomial $h(x)$, i.e.,
\begin{equation*}
  t_k=\frac{1}{2\pi}\int_{-\pi}^{\pi}h(x)e^{-i kx} \mathrm{d} x,\quad k=1-n,\dotsc,n-1.
\end{equation*}
Then, $h(x)=\sum_{k=1-n}^{n-1} t_ke^{i kx},$ is called  \emph{generating function} of $T_n$.
\end{definition}

\begin{lemma}\cite[p.\,13--15]{Chan:07} (the Grenander-Szeg\"{o} theorem)\label{lemma2.6}
Let $T_n$ be given in Definition \ref{De:gen-funct} with a generating function $h(x)$.
Then, the smallest and largest eigenvalues  $\lambda_{\min}(T_n)$ and $\lambda_{\max}(T_n)$, respectively, of $T_n$ are bounded as follows
\begin{equation*}
  h_{\min} \leqslant \lambda_{\min}(T_n) \leqslant \lambda_{\max}(T_n) \leqslant h_{\max},
\end{equation*}
with $h_{\min}$ and  $h_{\max}$  the minimum and maximum of $h(x)$, respectively.
In particular, if  $h_{\min}$ is positive, then $T_n$ is positive definite.
\end{lemma}

\begin{lemma}\cite{Xu:11}\label{lemma2.7}
Let $\{ q_j\}_{j=0}^\infty$ be a sequence of real numbers such that $q(\zeta)=\sum_{j=0}^\infty q_j\zeta^j$ is analytic in the unit disk $S=\{\zeta \in \mathbb{C}: |\zeta|\leqslant 1\}$.
Then for any positive integer $m$ and for any $\left(v^1,\ldots,v^m\right)$
\begin{equation*}
\begin{split}
\sum_{n=1}^m\left(\sum_{j=0}^{n-1} q_j v^{n-j},v^n\right)\geqslant 0,
\end{split}
\end{equation*}
if and only if
$\Real q(\zeta)\geqslant 0$, if and only if $\arg \left[q(\zeta)\right]\in\left[-\frac{\pi}{2},\frac{\pi}{2}\right]$.
\end{lemma}

There are already a class of new   multipliers  for the six-step BDF method  of the  time-dependent PDEs   \cite{ACYZ:20,CYZ:20}.
For example, the multiplier
$$\mu_1=13/9, \mu_2=-25/36, \mu_3=1/9, \mu_4=\mu_5=\mu_6=0$$ was constructed  in \cite{ACYZ:20} for  the  parabolic equation.
Here, based on the idea of  \cite{ACYZ:20,CYZ:20},  we develop the above  multiplier to the time fractional problem.

Taking $v^n=w^n-\mu_1w^{n-1}-\mu_2w^{n-2}-\mu_3w^{n-3}$, there exists
\begin{equation}\label{3.6}
\begin{split}
\sum_{j=0}^n g_j w^{n-j}
&=g_0 w^{n}+g_1 w^{n-1}+g_2 w^{n-2}+\cdots+g_{n-1} w^{1}\\
&=g_0\left(w^n-\mu_1w^{n-1}-\mu_2w^{n-2}-\mu_3w^{n-3}\right)\\
&\quad+\left(g_1+\frac{13}{9}g_0\right)\left(w^{n-1}-\mu_1w^{n-2}-\mu_2w^{n-3}-\mu_3w^{n-4}\right)\!+\!\cdots\\
&\quad+\left(g_{n-1}+\frac{13}{9}g_{n-2}+\cdots
+\frac{9^n(n-8)+8^{n+1}}{18^{n-1}}g_{0}\right)\\
&\quad\times\left(w^{1}-\mu_1w^{0}-\mu_2w^{-1}-\mu_3w^{-2}\right)
=\sum_{j=0}^{n-1} q_j v^{n-j}
\end{split}
\end{equation}
with the  starting values $w^{0}=w^{-1}=w^{-2}=0$ and
\begin{equation}\label{2.8}
\begin{split}
q_j=\sum_{l=0}^j\frac{9^{l+1}(l-7)+8^{l+2}}{18^{l}}g_{j-l}.
\end{split}
\end{equation}

\begin{lemma}\label{lemma4}
Let $q_j$ be defined by \eqref{2.8}. Then for any positive integer $m$, the following nonnegativity property holds
\begin{equation*}
\begin{split}
\sum_{n=1}^m\left(\sum_{j=0}^{n-1} q_j v^{n-j},v^n\right)\geqslant 0.
\end{split}
\end{equation*}
\end{lemma}
\begin{proof}
From \eqref{2.8}, we get
\begin{equation*}
\begin{split}
q(z)=\sum_{j=0}^\infty q_jz^j=&\frac{\left(\frac{147}{60}-6z+\frac{15}{2}z^2-\frac{20}{3}z^3+\frac{15}{4}z^4-\frac{6}{5}z^5+\frac{1}{6}z^6\right)^{\alpha}}{\left(1-\frac{1}{2}z\right)^2\left(1-\frac{4}{9}z\right)}\\
&=\frac{\left(1-z\right)^{\alpha}\left(\frac{147}{60}-\frac{213}{60}z+\frac{237}{60}z^2-\frac{163}{60}z^3+\frac{62}{60}z^4-\frac{10}{60}z^5\right)^{\alpha}}
{\left(1-\frac{1}{2}z\right)^2\left(1-\frac{4}{9}z\right)}.
\end{split}
\end{equation*}
Next we apply  the  Grenander-Szeg\"{o} theorem  to obtain the desired result.
Let  $z=e^{ix}$ with $x\in [0,\pi]$, we have
\begin{equation*}
\begin{split}
\left(1-z\right)^{\alpha}=\left(2\sin\frac{x}{2}\right)^{\alpha}e^{i{\alpha}\theta_1}
\end{split}
\end{equation*}
with $\theta_1=\arctan\frac{-\sin (x)}{1-\cos x}=\frac{x-\pi}{2}\leqslant 0$; and
\begin{equation*}
\begin{split}
\left(\frac{147}{60}-\frac{213}{60}z+\frac{237}{60}z^2-\frac{163}{60}z^3+\frac{62}{60}z^4-\frac{10}{60}z^5\right)^{\alpha}
=\left(a_6-ib_6\right)^{\alpha}=\left(a_6^2+b_6^2\right)^{\frac{\alpha}{2}}e^{i{\alpha}\theta_2}
\end{split}
\end{equation*}
with
\begin{equation*}
\begin{split}
a_6(x)&=\frac{1}{60}\left(147-213\cos (x)+237\cos(2x)-163\cos(3x)+62\cos(4x)-10\cos(5x)\right),\\
b_6(x)&=\frac{1}{60}\left(213\sin (x)-237\sin(2x)+163\sin(3x)-62\sin(4x)+10\sin(5x)\right)\geqslant 0,\\
\end{split}
\end{equation*}
and $\theta_2=\arctan\frac{-b_6(x)}{a_6(x)}\leqslant 0,~~a_6(x)\geqslant 0,$ or $\theta_2=\arctan\frac{-b_6(x)}{a_6(x)}-\pi\leqslant 0,~~a_6(x)\leqslant 0.$
Furthermore, there exists
\begin{equation*}
\begin{split}
\frac{1}{\left(1-\frac{1}{2}z\right)^2}=\frac{1}{\frac{5}{4}-\cos (x)}e^{i\theta_3},~~\theta_3=2\arctan\frac{\frac{1}{2}\sin (x)}{1-\frac{1}{2}\cos (x)}\geqslant 0,
\end{split}
\end{equation*}
and
\begin{equation*}
\begin{split}
\frac{1}{1-\frac{4}{9}z}=\left(\frac{97}{81}-\frac{8}{9}\cos (x)\right)^{-\frac{1}{2}}e^{i\theta_4},~~\theta_4=\arctan\frac{\frac{4}{9}\sin (x)}{1-\frac{4}{9}\cos (x)}\geqslant 0.
\end{split}
\end{equation*}
From Lemma \ref{lemma2.7}, we need to prove
\begin{equation*}
\begin{split}
\Real\left\{\frac{\left(\frac{147}{60}-6z+\frac{15}{2}z^2-\frac{20}{3}z^3+\frac{15}{4}z^4-\frac{6}{5}z^5+\frac{1}{6}z^6\right)^{\alpha}}{\left(1-\frac{1}{2}z\right)^2\left(1-\frac{4}{9}z\right)}\right\}\geqslant0,
\end{split}
\end{equation*}
which is equal to prove
\begin{equation*}
\begin{split}
\arg\left\{\frac{\left(\frac{147}{60}-6z+\frac{15}{2}z^2-\frac{20}{3}z^3+\frac{15}{4}z^4-\frac{6}{5}z^5+\frac{1}{6}z^6\right)^{\alpha}}{\left(1-\frac{1}{2}z\right)^2\left(1-\frac{4}{9}z\right)}\right\}\in\left[-\frac{\pi}{2},\frac{\pi}{2}\right].
\end{split}
\end{equation*}
According to the above equations, we have
\begin{equation*}
\begin{split}
&\arg\left\{\frac{\left(\frac{147}{60}-6z+\frac{15}{2}z^2-\frac{20}{3}z^3+\frac{15}{4}z^4-\frac{6}{5}z^5+\frac{1}{6}z^6\right)^{\alpha}}{\left(1-\frac{1}{2}z\right)^2\left(1-\frac{4}{9}z\right)}\right\}\\
&=\arg\left\{\left(1-z\right)^{\alpha}\right\}+\arg\left\{\left(\frac{147}{60}-\frac{213}{60}z+\frac{237}{60}z^2-\frac{163}{60}z^3+\frac{62}{60}z^4-\frac{10}{60}z^5\right)^{\alpha}\right\}\\
&\quad+\arg\left\{\frac{1}{\left(1-\frac{1}{2}z\right)^2}\right\}+\arg\left\{\frac{1}{1-\frac{4}{9}z}\right\}
=\alpha\theta_1+\alpha\theta_2+\theta_3+\theta_4.
\end{split}
\end{equation*}

We shall prove  $\alpha \theta_1+\alpha \theta_2+\theta_3+\theta_4\leqslant \theta_3+\theta_4<\frac{\pi}{2}$. Let $\delta(x)=(\theta_3+\theta_4)(x)$, we have
\begin{equation*}
\begin{split}
\delta'(x)=\frac{2}{\left(97-72y\right)\left(5-4y\right)}p(y) ~~{\rm with}~y=\cos (x).
\end{split}
\end{equation*}
Here $p(y)=-216y^2+388y-137$ with the roots $y_1=\frac{-97-\sqrt{2011}}{-108}>1$ and $y_2=\frac{-97+\sqrt{2011}}{-108}\approx0.48292$ with $x_2\approx1.0668$.
In further, we obtain $p(y)<0$ if $y\in(-1,y_2)$ and $p(y)>0$ if $y\in(y_2,1)$.
Moreover, combining with $\left(97-72y\right)\left(5-4y\right)>0$, it implies  that
 $\delta'(x)<0$ if $x\in(x_2,\pi)$ and $\delta'(x)>0$ if $x\in(0,x_2)$.
Therefore, the function  $\delta$ attains its maximum at $x^\star=x_2$
and
$$\delta(x^\star)=2\arctan\frac{\frac{1}{2}\sin (x_2)}{1-\frac{1}{2}\cos (x_2)}+\arctan\frac{\frac{4}{9}\sin (x_2)}{1-\frac{4}{9}\cos (x_2)}<1.51<\frac{\pi}{2}.$$

On the other hand, since $\alpha \theta_1+\alpha \theta_2+\theta_3+\theta_4\geqslant \theta_1+\theta_2+\theta_3+\theta_4$. We  just need to prove $\theta_1+\theta_2+\theta_3+\theta_4+\theta_5\geqslant -\frac{\pi}{2}$.
That is to say, we need to prove
\begin{equation}\label{ad3.3}
\begin{split}
\Real\left\{\frac{\left(\frac{147}{60}-6z+\frac{15}{2}z^2-\frac{20}{3}z^3+\frac{15}{4}z^4-\frac{6}{5}z^5+\frac{1}{6}z^6\right)}{\left(1-\frac{1}{2}z\right)^2\left(1-\frac{4}{9}z\right)}\right\}\geqslant0.
\end{split}
\end{equation}
Fortunately, the result \eqref{ad3.3} has been proved in Proposition 2.1 of \cite{ACYZ:20}.
The proof is completed.
\end{proof}

\begin{lemma}\label{lemma3}
For any positive integer $m$, it holds that
\begin{equation*}
\sum_{n=1}^{m}\left( \nabla w^n,\nabla w^n-\sum_{j=1}^6 \mu_j \nabla w^{n-j}\right)\geqslant \frac{1}{32}\sum_{n=1}^{m}\|\nabla w^n\|^2.
\end{equation*}
\end{lemma}
\begin{proof}
With this notation $\mu_0:=-31/32$, $\mu_1:=13/9$, $\mu_2:=-25/36$, $\mu_3:=1/9$, $\mu_4=\mu_5=\mu_6=0$, it yields
\begin{equation*}
\begin{split}
\sum_{n=1}^{m}\left(\nabla w^n,\nabla w^n-\sum_{j=1}^6 \mu_j \nabla w^{n-j}\right)=\frac{1}{32}\sum_{n=1}^{m}\|\nabla w^n\|^2+\sum_{i,j=1}^m \ell_{i,j}\left(\nabla w^i,\nabla w^j\right).
\end{split}
\end{equation*}
To this end, we introduce the lower triangular Toeplitz matrix $L_2=(\ell_{ij})\in \Re^{m,m}$  with entries
\[\ell_{i,i-j}=-\mu_j, \quad j=0,1,2,3, \quad i=j+1,\dotsc,m,\]
and all other entries equal zero.
According to Definition \ref{De:gen-funct},  the generating function of $(L_2+L_2^T)/2$ is
\begin{equation*}
\begin{split}
h(x)&=\frac{31}{32} -\frac{13}{9}\cos(x)+\frac{25}{36}\cos(2x)-\frac{1}{9}\cos(3x)\\
&=-\frac{4}{9}\cos^3(x)+\frac{25}{18}\cos^2(x)-\frac{10}{9}\cos(x)+\frac{79}{288}, \quad \forall x \in \Re.
\end{split}
\end{equation*}
Hence, we consider the polynomial $p,$
\begin{equation*}
p(s):=-\frac{4}{9}s^3+\frac{25}{18}s^2-\frac{10}{9}s+\frac{79}{288},\quad s\in [-1,1].
\end{equation*}
It is easily seen that $p$ attains its minimum at $s^\star=(25-\sqrt{145})/24$ 
and
\[p(s^\star)>0.009321552602567 > 0.\]
Using Lemma \ref{lemma2.4} and  \ref{lemma2.6}, it implies that $L_2$ is positive definite.  Then we obtain
\begin{equation*}
\begin{split}
\sum_{n=1}^{m}\left( \nabla w^n,\nabla w^n-\sum_{j=1}^6 \mu_j \nabla w^{n-j}\right)\geqslant \frac{1}{32}\sum_{n=1}^{m}\|\nabla w^n\|^2.
\end{split}
\end{equation*}
The proof is completed.
\end{proof}

\begin{lemma}\cite[p.\,14]{Quarteroni:08} (Discrete Gronwall Lemma)\label{lemma5}
Assume that $h_n$ is a non-negative sequence, and that the sequence $\phi_n$ satisfies
\begin{equation*}
\left\{ \begin{array}
 {l@{\quad} l}
\phi_0\leqslant\psi_0,\\
   \phi_n\leqslant\psi_0+\sum_{s=0}^{n-1}p_s+\sum_{s=0}^{n-1}h_s\phi_s,~~n\geqslant1.
 \end{array}
 \right.
\end{equation*}
Then, if $\psi_0\geqslant0$ and $p_n\geqslant0$ for $n\geqslant0$, $\phi_n$ satisfies
\begin{equation*}
\phi_n\leqslant\left(\psi_0+\sum_{s=0}^{n-1}p_s\right)\exp\left(\sum_{s=0}^{n-1}h_s\right),~~n\geqslant1.
\end{equation*}
\end{lemma}
The above technical Lemmas play an important role in the error analysis and we shall frequently use the discrete Gronwall Lemma. The well-posedness and the limited regularity of the time-fractional Allen-Cahn equation \eqref{3.1} with $F$ satisfying \eqref{3.3} was studied in \cite{Jin:18} for the nonlinear subdiffusion equation. It is proved in Theorem 3.1 of \cite{Jin:18}  that if $u_0\in H^1_0(\Omega)\cap H^2(\Omega)$, then \eqref{3.1} admits a unique solution $u$  satisfying
\begin{equation}\label{3.10}
u\in C^{\alpha}\left([0,T];L^2(\Omega)\right)\cap C\left([0,T];H^1_0(\Omega)\cap H^2(\Omega)\right),~~ \partial_tu\in L^2(\Omega).
\end{equation}

\section{Error analysis for BDF$6$ SAV schemes}\label{Se:error}
In this section, we shall carry out
error analysis of the BDF$6$ SAV schemes for the time-fractional Allen-Cahn equation described
as in \eqref{3.5}, \eqref{2.5b}, \eqref{2.5c} and \eqref{2.5d}.
We denote hereafter $\bar{e}^n :=\bar{w}^n-w(t_n)=\bar{u}^n-u(t_n),~e^n :=w^n-w(t_n)=u^n-u(t_n),~s^n :=r^n-r(t_n)$.
\begin{theorem} \label{theorem3}	
Given initial condition $\bar{u}^0=u^0=u(0),~r^0=E[u^0]$. Let $\bar{u}^n$ and $u^n$ be computed with
the BDF$6$ SAV schemes \eqref{3.5}, \eqref{2.5b}, \eqref{2.5c} and \eqref{2.5d}.
If $u(t)$, $\partial_t^{\alpha+6}u(t)$ and their Fourier transforms belong to $L_1(\Re)$ and the following conditions hold
\begin{equation*}
\begin{split}
u_0\in H^1_0(\Omega)\cap H^2(\Omega),~~\partial_t^6u\in L^2(0,T;L^2(\Omega)),~~\partial_t^iu\in L^2(0,T;H^1(\Omega)),i=1,2.
\end{split}
\end{equation*}
Then for $n\tau\leqslant T$ and $\tau\leqslant\frac{1}{1+T^2C_0^8}$, we have
\begin{equation*}
\|\bar{e}^n\|_{H^2},\|e^n\|_{H^2}\leqslant Ct_n^{-1}\tau^6,
\end{equation*}
where the constants $C_0$ and $C$  are dependent on $T,\Omega$ and the exact
solution $u$ but are independent of $\tau$.
\end{theorem}
\begin{proof}
The main task is to prove
\begin{equation}\label{3.7}	
|1-\xi^q|\leqslant C_0\tau,\forall q\leqslant N,
\end{equation}
where the constant $C_0$ is dependent on $T,\Omega$ and the exact
solution $u$ but is independent of $\tau$, and will
be defined in the proof process. Below we use the mathematical induction to  prove \eqref{3.7}.

It is trivial that the claimed inequality \eqref{3.7} holds for $q=0$. For $1\leqslant m\leqslant N$, assume that
\begin{equation}\label{3.8}	
|1-\xi^q|\leqslant C_0\tau,\forall q\leqslant m-1.
\end{equation}
It remains to prove that
\begin{equation}\label{3.9}	
|1-\xi^m|\leqslant C_0\tau.
\end{equation}
{\bf Step 1}. First, we prove the $H^1$ bound for $\bar{u}^{n-1}$ and $u^{n-1}$ for all $n\leqslant m\leqslant N$. From \eqref{2.13}, we obtain
\begin{equation*}	
\|\nabla u^q\|^2=(\nabla u^q,\nabla u^q)=(\mathcal{L}u^q,u^q)\leqslant M^2,\forall q\leqslant N.
\end{equation*}
According to the induction hypothesis \eqref{3.8}, \eqref{2.5d} and the above inequality, if we choose $\tau$ small enough such that
$\tau\leqslant \frac{1}{2C_0^8}$,
we have
\begin{equation*}
\left|\eta^q\right|=\left|1-(1-\xi^q)^8\right|\geqslant1-|1-\xi^q|^8\geqslant1-\frac{\tau^7}{2},\forall q\leqslant m-1,
\end{equation*}
and
\begin{equation*}	
\|\nabla \bar{u}^q\|\leqslant\left|\eta^q\right|^{-1}\|\nabla u^q\|\leqslant 2M,\forall q\leqslant m-1,\forall \tau\leqslant 1.
\end{equation*}
{\bf Step 2}. Then, we estimate $\|\bar{e}^n\|_{H^2}$ for all $1\leqslant n\leqslant m\leqslant N$. According to \eqref{3.10} and the above inequality, we choose $C$
large enough such that
\begin{equation}\label{3.19}		
\left\|u(t)\right\|_{H^2}\leqslant C,~~\forall t\leqslant T,\quad
\left\|\nabla \bar{u}^q\right\|\leqslant C,~~ \forall q\leqslant m-1.
\end{equation}
From \eqref{3.5} and \eqref{3.4}, we can write down the error equation as
\begin{equation}\label{3.21}		
\tau^{-\alpha}\sum_{j=0}^n g_j\bar{e}^{n-j}-\Delta\bar{e}^n=R^n+\tau^{-\alpha}\sum_{j=1}^n g_j(\bar{u}^{n-j}-u^{n-j})+Q^n,
\end{equation}
where $R^n,~Q^n$ are given by
\begin{equation}\label{3.22}		
R^n=\partial_t^\alpha w(t_n)-\tau^{-\alpha}\sum_{j=0}^n g_j w(t_{n-j}),
\end{equation}
and
\begin{equation}\label{3.23}		
Q^n=f\left[u(t_n)\right]-f\left[B_6(\bar{u}^{n-1})\right].
\end{equation}
Taking the inner product of \eqref{3.21} with $\bar{z}^n=\bar{e}^n-\sum_{i=1}^3\mu_i\bar{e}^{n-i}$, then multiplying by $\tau$ and summing up for $n$ from $1$ to $ m$, we get
\begin{equation*}
\begin{split}		
&\tau^{1-\alpha}\sum_{n=1}^m\left(\sum_{j=0}^{n-1} q_j\bar{z}^{n-j},\bar{z}^n\right)+\tau\sum_{n=1}^m\left(\nabla\bar{e}^n,\nabla\bar{z}^n\right)\\
&=\tau\sum_{n=1}^m\left(R^n+\tau^{-\alpha}\sum_{j=1}^n g_j(\bar{u}^{n-j}-u^{n-j})-Q^n,\bar{z}^n\right),
\end{split}
\end{equation*}
where \eqref{3.6} is employed on the first term of the left hand side.
According to Lemma \ref{lemma4}, it implies
\begin{equation*}	
\begin{split}
&\tau\sum_{n=1}^m\left(\nabla\bar{e}^n,\nabla\bar{e}^n-\sum_{i=1}^3\mu_i\nabla\bar{e}^{n-i}\right)\\
&\leqslant\tau\sum_{n=1}^m\left(\left\|R^n\right\|+\tau^{-\alpha}\left\|\sum_{j=1}^n g_j(\bar{u}^{n-j}-u^{n-j})\right\|+\left\|Q^n\right\|\right)\left(\left\|\bar{e}^n\right\|+\sum_{i=1}^3|\mu_i|\left\|\bar{e}^{n-i}\right\|\right).
\end{split}
\end{equation*}
According to Lemma \ref{lemma3} and \eqref{2.5d}, we obtain
\begin{equation*}
\begin{split}		
&\frac{1}{32}\tau\sum_{n=1}^m\left\|\nabla\bar{e}^n\right\|^2\\
&\leqslant C\tau\sum_{n=1}^m\left(\left\|R^n\right\|+\tau^{-\alpha}\left\|\sum_{j=1}^n g_j(1-\eta^{n-j})\bar{u}^{n-j}\right\|+\left\|Q^n\right\|\right)\left(\left\|\nabla\bar{e}^n\right\|+\sum_{i=1}^3|\mu_i|\|\nabla\bar{e}^{n-i}\|\right).
\end{split}
\end{equation*}
Suppose $l$ is chosen so that $\|\nabla\bar{e}^l\|=\max \limits_{1\leqslant n\leqslant m}\|\nabla\bar{e}^n\|$. Then
\begin{equation*}
\begin{split}		
&\frac{1}{32}m\tau \left\|\nabla\bar{e}^l\right\|^2\\
&\leqslant C\tau\sum_{n=1}^m\left(\left\|R^n\right\|+\tau^{-\alpha}\left\|\sum_{j=1}^n g_j(1-\eta^{n-j})\bar{u}^{n-j}\right\|+\left\|Q^n\right\|\right)\left(1+\sum_{i=1}^3|\mu_i|\right)\left\|\nabla\bar{e}^l\right\|,
\end{split}
\end{equation*}
whence
\begin{equation}\label{3.27}
\begin{split}		
\frac{1}{32}t_m \left\|\nabla\bar{e}^l\right\|
\leqslant C\tau\sum_{n=1}^m\left(\left\|R^n\right\|+\tau^{-\alpha}\left\|\sum_{j=1}^n g_j(1-\eta^{n-j})\bar{u}^{n-j}\right\|+\left\|Q^n\right\|\right).
\end{split}
\end{equation}
In the following, we bound the right hand side of \eqref{3.27}. From \eqref{3.22} and \cite{Chendeng:13}, we get
\begin{equation}\label{3.28}		
\tau\sum_{n=1}^m\left\|R^n\right\|\leqslant C\tau\sum_{n=1}^m\left\|\mathcal{F}[\partial_t^{\alpha+6}w]\right\|_{L^1}\cdot\tau^6 \leqslant C T\tau^6\left\|\mathcal{F}[\partial_t^{\alpha+6}w]\right\|_{L^1}.
\end{equation}
From \eqref{2.5d} and the induction assumption \eqref{3.8}, we obtain
\begin{equation*}
\left|\eta^q-1\right|=\left|1-\xi^q\right|^8\leqslant C_0^8\tau^8, ~~\forall q\leqslant n-1.
\end{equation*}
According to the above inequality and \eqref{3.19}, it yields
\begin{equation} \label{3.30}
\begin{split}	
\tau\sum_{n=1}^m\tau^{-\alpha}\left\|\sum_{j=1}^n g_j(1-\eta^{n-j})\bar{u}^{n-j}\right\|&\leqslant C\tau\sum_{n=1}^m\tau^{-\alpha}C_0^8\tau^8\sum_{j=1}^n \left\|\nabla\bar{u}^{n-j}\right\|\\
 &\leqslant CC_0^8\tau^{9-\alpha}\sum_{n=1}^m n\leqslant CT^2C_0^8\tau^6.
\end{split}	
\end{equation}
From \eqref{3.23} and \eqref{3.3}, we derive
\begin{equation*}
\begin{split}	
\left|Q^n\right|&\leqslant\left|f\left[B_6(\bar{u}^{n-1})\right]-f\left[B_6(u(t_{n-1}))\right]\right|+\left|f\left[B_6(u(t_{n-1}))\right]-f\left[u(t_n)\right]\right|\\
&\leqslant L\left|B_6(\bar{e}^{n-1})\right|+L\left|B_6(u(t_{n-1}))-u(t_{n})\right|\\
&=L\left|B_6(\bar{e}^{n-1})\right|+L\left|\sum_{i=1}^6b_i\int_{t^{n-i}}^{t^n}\left(t^{n-i}-s\right)^5\partial_t^6u(s)ds\right|,
\end{split}
\end{equation*}
where
$b_1=-\frac{6}{5!},b_2=\frac{15}{5!},b_3=-\frac{20}{5!},b_4=\frac{15}{5!},b_5=-\frac{6}{5!},b_6=\frac{1}{5!}$ are determined by Taylor expansion.
\begin{equation}\label{3.32}	
\tau\sum_{n=1}^m\left\|Q^n\right\|\leqslant C\tau\sum_{n=1}^m\left\|\bar{e}^{n-1}\right\|+C\tau^6\int_{0}^{T}\left\|\partial_t^6u(s)\right\|ds.
\end{equation}
Now, combining \eqref{3.27}, \eqref{3.28}, \eqref{3.30}, \eqref{3.32}, we get
\begin{equation}\label{3.33}
\begin{split}		
t_m\left\|\nabla\bar{e}^m\right\|
\leqslant C\tau\sum_{n=1}^m\left\|\nabla\bar{e}^{n-1}\right\|\!+\!C\tau^6\left(T\left\|\mathcal{F}[\partial_t^{\alpha+6}w]\right\|_{L^1}\!+\!T^2C_0^8\!+\!\int_{0}^{T}\left\|\partial_t^6u(s)\right\|ds\right).
\end{split}
\end{equation}
Similarly, the estimate for $\left\|\Delta\bar{e}^{m}\right\|$ can be obtained by using the same procedure. In fact,
taking the inner product of \eqref{3.21} with $\bar{v}^n=-\Delta\bar{e}^n+\sum_{i=1}^3\mu_i\Delta\bar{e}^{n-i}$, then multiplying by $\tau$ and summing up for $n$ from $1$ to $ m$, we get
\begin{equation*}
\begin{split}		
&\tau^{1-\alpha}\sum_{n=1}^m\left(\sum_{j=0}^{n-1} q_j\nabla\bar{z}^{n-j},\nabla\bar{z}^n\right)+\tau\sum_{n=1}^m\left(-\Delta\bar{e}^n,\bar{v}^n\right)\\
&=\tau\sum_{n=1}^m\left(R^n+\tau^{-\alpha}\sum_{j=1}^n g_j(\bar{u}^{n-j}-u^{n-j})-Q^n,\bar{v}^n\right),
\end{split}
\end{equation*}
where \eqref{3.6} is utilized on the first term of the left hand side.
According to Lemma \ref{lemma4}, it yields
\begin{equation*}
\begin{split}		
&\tau\sum_{n=1}^m\left(-\Delta\bar{e}^n,-\Delta\bar{e}^n+\sum_{i=1}^3\mu_i\Delta\bar{e}^{n-i}\right)\\
&\leqslant \tau\sum_{n=1}^m\left(\left\|R^n\right\|+\tau^{-\alpha}\left\|\sum_{j=1}^n g_j(\bar{u}^{n-j}-u^{n-j})\right\|+\left\|Q^n\right\|\right)\left(\left\|\Delta\bar{e}^n\right\|+\sum_{i=1}^3|\mu_i|\|\Delta\bar{e}^{n-i}\|\right).
\end{split}
\end{equation*}
According to Lemma \ref{lemma3} and \eqref{2.5d}, it implies
\begin{equation*}
\begin{split}		
&\frac{1}{32}\tau\sum_{n=1}^m\left\|\Delta\bar{e}^n\right\|^2\\
&\leqslant \tau\sum_{n=1}^m\left(\left\|R^n\right\|+\tau^{-\alpha}\left\|\sum_{j=1}^n g_j(1-\eta^{n-j})\bar{u}^{n-j}\right\|+\left\|Q^n\right\|\right)\left(\left\|\Delta\bar{e}^n\right\|+\sum_{i=1}^3|\mu_i|\|\Delta\bar{e}^{n-i}\|\right).
\end{split}
\end{equation*}
Suppose $l$ is chosen so that $\|\Delta\bar{e}^l\|=\max \limits_{1\leqslant n\leqslant m}\|\Delta\bar{e}^n\|$. Then
\begin{equation*}
\begin{split}		
&\frac{1}{32}m\tau \left\|\Delta\bar{e}^l\right\|^2\\
&\leqslant \tau\sum_{n=1}^m\left(\left\|R^n\right\|+\tau^{-\alpha}\left\|\sum_{j=1}^n g_j(1-\eta^{n-j})\bar{u}^{n-j}\right\|+\left\|Q^n\right\|\right)\left(1+\sum_{i=1}^3|\mu_i|\right)\left\|\Delta\bar{e}^l\right\|,
\end{split}
\end{equation*}
whence
\begin{equation*}
\begin{split}		
\frac{1}{32}t_m \left\|\Delta\bar{e}^l\right\|
\leqslant C\tau\sum_{n=1}^m\left(\left\|R^n\right\|+\tau^{-\alpha}\left\|\sum_{j=1}^n g_j(1-\eta^{n-j})\bar{u}^{n-j}\right\|+\left\|Q^n\right\|\right).
\end{split}
\end{equation*}
Combining \eqref{3.28}, \eqref{3.30}, \eqref{3.32}, it yields
\begin{equation*}
\begin{split}		
t_m\left\|\Delta\bar{e}^m\right\|
\leqslant C\tau\sum_{n=1}^m\left\|\Delta\bar{e}^{n-1}\right\|+C\tau^6\left(T\left\|\mathcal{F}[\partial_t^{\alpha+6}w]\right\|_{L^1}+T^2C_0^8+\int_{0}^{T}\left\|\partial_t^6u(s)\right\|ds\right).
\end{split}
\end{equation*}
From \eqref{3.33} and the above inequality, we derive
\begin{equation*}
\begin{split}		
\left\|\bar{e}^m\right\|_{H^2}
\leqslant Ct_m^{-1}\tau\sum_{n=0}^{m-1}\left\|\bar{e}^{n}\right\|_{H^2}\!+\!Ct_m^{-1}\tau^6\left(T\left\|\mathcal{F}[\partial_t^{\alpha+6}w]\right\|_{L^1}\!+\!T^2C_0^8\!+\!\int_{0}^{T}\left\|\partial_t^6u(s)\right\|ds\!\right).
\end{split}
\end{equation*}
Applying the discrete Gronwall Lemma \ref{lemma5} to the above inequality, we get
\begin{equation*}
\begin{split}		
\left\|\bar{e}^m\right\|_{H^2}
&\leqslant Ct_m^{-1}\tau^6\left(T\left\|\mathcal{F}[\partial_t^{\alpha+6}w]\right\|_{L^1}+T^2C_0^8+\int_{0}^{T}\left\|\partial_t^6u(s)\right\|ds\right)\\
&\leqslant C_2\left(1+T^2C_0^8\right)t_m^{-1}\tau^6,
\end{split}
\end{equation*}
where $C_2:=C\max\left(T\left\|\mathcal{F}[\partial_t^{\alpha+6}w]\right\|_{L^1}+\int_{0}^{T}\left\|\partial_t^6u(s)\right\|ds,1\right)$ is independent of $\tau$ and $C_0$.
In particular, the above inequality implies
\begin{equation}\label{3.41}
\begin{split}		
\left\|\bar{e}^n\right\|_{H^2}
\leqslant C_2\left(1+T^2C_0^8\right)t_n^{-1}\tau^6,~~\forall 1\leqslant n \leqslant m.
\end{split}
\end{equation}
Combining \eqref{3.19} and \eqref{3.41}, we obtain
\begin{equation}\label{3.42}
\begin{split}		
\left\|\bar{u}^n\right\|_{H^2}
\leqslant C_2\left(1+T^2C_0^8\right)t_n^{-1}\tau^6+C\leqslant C_2\left(1+T^2C_0^8\right)+C:=\bar{C},~~\forall \tau\leqslant1.
\end{split}
\end{equation}
{\bf Step 3}. Next, we estimate $\left|1-\xi^m\right|$. By direct calculation,
\begin{equation}\label{3.44}
r_{tt}=\int_\Omega\left(\left|\nabla u_t\right|^2+\nabla u\cdot\nabla u_{tt}+f'(u)u_t^2+f(u)u_{tt}\right)dx.
\end{equation}
From \eqref{2.5d} and \eqref{2.2}, it yields
\begin{equation}\label{3.45}
s^n-s^{n-1}=\tau\left(\mathcal{K}[u(t_n)]-\frac{r^n}{E(\bar{u}^n)}\mathcal{K}(\bar{u}^n)\right)+J^n,
\end{equation}
where
\begin{equation}\label{3.46}
\begin{split}
\mathcal{K}[u(t_n)]&=-\int_\Omega\left(-\Delta u(t_n)+f[u(t_n)]\right)\left(\frac{u(t_n)-u(t_{n-1})}{\tau}+\mathcal{O}(\tau)\right)dx\\
\mathcal{K}(\bar{u}^n)&=-\int_\Omega\left(-\Delta \bar{u}^n+f(\bar{u}^n)\right)\frac{\bar{u}^n-\bar{u}^{n-1}}{\tau}dx\\
J^n&=r(t_{n-1})-r(t_{n})+\tau r_t(t_n)=\int_{t_{n-1}}^{t_n}\left(s-t_{n-1}\right)r_{tt}(s)ds.
\end{split}
\end{equation}
Taking the sum of \eqref{3.45} for $n$ from $1$ to $m$ and noting  $s^0=0$, we have
\begin{equation}\label{3.47}
s^m=\tau\sum_{n=1}^{m}\left(\mathcal{K}[u(t_n)]-\frac{r^n}{E(\bar{u}^n)}\mathcal{K}(\bar{u}^n)\right)+\sum_{n=1}^{m}J^n.
\end{equation}
Now, we bound the terms on the right hand side of \eqref{3.47}. From \eqref{3.46}, \eqref{3.44}, \eqref{3.3} and \eqref{3.19}, we have
\begin{equation}\label{3.48}
\left|J^n\right|\leqslant C\tau\int_{t_{n-1}}^{t_n}\left|r_{tt}\right|ds\leqslant C\tau\int_{t_{n-1}}^{t_n}\left(\left\|u_t(s)\right\|_{H^1}^2+\left\|u_{tt}(s)\right\|_{H^1}\right)ds.
\end{equation}
Next,
\begin{equation}\label{3.49}
\begin{split}
&\left|\mathcal{K}[u(t_n)]-\frac{r^n}{E(\bar{u}^n)}\mathcal{K}(\bar{u}^n)\right|\\
&\leqslant \mathcal{K}[u(t_n)]\left|1-\frac{r^n}{E(\bar{u}^n)}\right|+\frac{r^n}{E(\bar{u}^n)}\left|\mathcal{K}[u(t_n)]-\mathcal{K}(\bar{u}^n)\right|:=P_1+
P_2.
\end{split}
\end{equation}
From \eqref{3.49}, \eqref{3.46}, \eqref{3.10}, $E(v)>\b{C}>0,~\forall v$ and \eqref{2.12}, it holds
\begin{equation}\label{3.50}
\begin{split}
P_1&=\mathcal{K}[u(t_n)]\left|1-\frac{r^n}{E(\bar{u}^n)}\right|\leqslant C\left|1-\frac{r^n}{E(\bar{u}^n)}\right|\\
&\leqslant C\left|\frac{r(t_n)}{E\left[u(t_n)\right]}-\frac{r^n}{E\left[u(t_n)\right]}\right|+ C\left|\frac{r^n}{E\left[u(t_n)\right]}-\frac{r^n}{E(\bar{u}^n)}\right|\\
&\leqslant C\left(\left|s^n\right|+\left|E\left[u(t_n)\right]-{E(\bar{u}^n)}\right|\right).
\end{split}
\end{equation}
According to \eqref{3.49}, \eqref{2.12}, $E(v)>\b{C}>0,~\forall v$, \eqref{3.46}, \eqref{3.3},  \eqref{3.42} and \eqref{3.10}, we derive
\begin{equation}\label{3.51}
\begin{split}
P_2&=\frac{r^n}{E(\bar{u}^n)}\left|\mathcal{K}[u(t_n)]-\mathcal{K}(\bar{u}^n)\right|\leqslant C\left|\mathcal{K}[u(t_n)]-\mathcal{K}(\bar{u}^n)\right|\\
&\leqslant C\int_\Omega\left|\left(-\Delta\bar{e}^n+f(\bar{u}^n)-f[u(t_n)]\right)\frac{\bar{u}^n-\bar{u}^{n-1}}{\tau}\right|dx\\
& \quad +C\int_\Omega\left|\left(-\Delta u(t_n)+f[u(t_n)]\right)\left(\frac{\bar{e}^n-\bar{e}^{n-1}}{\tau}+\mathcal{O}(\tau)\right)\right|dx\\
&\leqslant C\tau^{-1}\left(\left\|\Delta\bar{e}^n\right\|+\left\|\bar{e}^n\right\|\right)\leqslant CC_2\left(1+T^2C_0^8\right)t_n^{-1}\tau^5.
\end{split}
\end{equation}
On the other hand,
\begin{equation}\label{3.52}
\begin{split}
\left|E\left[u(t_n)\right]-{E(\bar{u}^n)}\right|&\leqslant\frac{1}{2}\left(\left\|\nabla u(t_n)\right\|+\left\|\nabla \bar{u}^n\right\|\right)\left\|\nabla u(t_n)-\nabla \bar{u}^n\right\|\\
& \quad+\int_\Omega \left|F[u(t_n)]-F(\bar{u}^n)\right|dx\\
&\leqslant C\left(\left\|\nabla\bar{e}^n\right\|+\left\|\bar{e}^n\right\|\right)\leqslant CC_2\left(1+T^2C_0^8\right)t_n^{-1}\tau^6.
\end{split}
\end{equation}
From \eqref{3.47}, \eqref{3.48}, \eqref{3.49}, \eqref{3.50}, \eqref{3.51} and \eqref{3.52}, we derive
\begin{equation*}
\begin{split}
\left|s^m\right|&\leqslant \tau\sum_{n=1}^{m}\left|\mathcal{K}(u(t_n))-\frac{r^n}{E(\bar{u}^n)}\mathcal{K}(\bar{u}^n)\right|+\sum_{n=1}^{m}\left|J^n\right|\\
&\leqslant C\tau\sum_{n=1}^{m}\left|s^n\right|+CC_2\left(1+T^2C_0^8\right)\tau^4+C\tau\int_0^T\left(\left\|u_t(s)\right\|_{H^1}^2+\left\|u_{tt}(s)\right\|_{H^1}\right)ds\\
&\leqslant C\tau\sum_{n=1}^{m-1}\left|s^n\right|+CC_2\left(1+T^2C_0^8\right)\tau^4+C\tau.
\end{split}
\end{equation*}
Applying the discrete Gronwall lemma to the above inequality, we obtain
\begin{equation}\label{3.53}
\begin{split}
\left|s^m\right|\leqslant CC_2\left(1+T^2C_0^8\right)\tau+C\tau.
\end{split}
\end{equation}
From \eqref{2.5c}, \eqref{3.50}, \eqref{3.52}, \eqref{3.53}, we have
\begin{equation*}
\begin{split}
\left|1-\xi^m\right|&=\left|1-\frac{r^m}{E(\bar{u}^m)}\right|\leqslant C\left(\left|E\left[u(t_m)\right]-{E(\bar{u}^m)}\right|+\left|s^m\right|\right)\\
&\leqslant CC_2\left(1+T^2C_0^8\right)\tau^5+CC_2\left(1+T^2C_0^8\right)\tau^4+C\tau\\
&\leqslant C_3\tau \left(\left(1+T^2C_0^8\right)\tau^3+1\right)\leqslant C_3 \tau \left(\left(1+T^2C_0^8\right)\tau+1\right)\leqslant C_0\tau.
\end{split}
\end{equation*}
where we choose $C_0=2C_3$, $\tau\leqslant\frac{1}{1+T^2C_0^8}$ and the constant $C_3$ is independent of $C_0$ and $\tau$.
The induction process for \eqref{3.7} is finished.

 Finally,  it remains to show $\left\|e^m\right\|_{H^2}\leqslant Ct_m^{-1}\tau^6$.
From \eqref{2.5d} and \eqref{3.42}, we derive
\begin{equation*}
\left\|u^m-\bar{u}^m\right\|_{H^2}\leqslant \left|\eta^m-1\right|\left\|\bar{u}^m\right\|_{H^2}\leqslant\left|\eta^m-1\right|\bar{C}.
\end{equation*}
From \eqref{2.5d} and \eqref{3.7}, it yields
\begin{equation*}
\left|\eta^m-1\right|=\left|1-\xi^m\right|^8\leqslant C_0^8\tau^8.
\end{equation*}
According to the triangle inequality, \eqref{3.41} and the above inequality, we obtain
\begin{equation*}
\left\|e^m\right\|_{H^2}\leqslant \left\|\bar{e}^m\right\|_{H^2}+\left\|u^m-\bar{u}^m\right\|_{H^2}\leqslant C_2\left(1+T^2C_0^8\right)t_m^{-1}\tau^6+\bar{C}C_0^8\tau^8\leqslant C t_m^{-1}\tau^6.
\end{equation*}
The proof is completed.
\end{proof}
\begin{remark}
Without detailed proof, a similar result for the usual Allen-Cahn equation with BDF$6$  SAV schemes  \cite{Huangg:20} can be obtained.
\end{remark}
\section{Numerical experiments}
We numerically verify the above theoretical results including convergent order by the $l_\infty$ norm and the discrete $L^2$-norm.
Without loss of generality, we add a force term on the right hand side of \eqref{3.1}.  In the test, we use the Legendre-Galerkin method \cite{Shen:11} with 50 modes  for space
discretization so that the spatial discretization error is negligible compared with the time discretization
error.
\begin{example}\label{example6.1}
Consider the one-dimensional time-fractional Allen-Cahn equation \eqref{3.1} on a finite domain $\Omega=(-1,1)$
 with the  initial condition  $u(x,0)=0.1(1-x^2)$  and the homogeneous Dirichlet boundary condition $u(1,t)=u(-1,t)=0$. The forcing function is chosen such that the exact solution is $u(x,t)=(t^{10}+0.1)(1-x^2)$.
\end{example}
\begin{table}[h]\fontsize{9.5pt}{12pt}\selectfont
 \begin{center}
  \caption {The $l_{\infty}$ norm and discrete $L^{2}$-norm  for  BDF$6$ SAV schemes.} \vspace{5pt}
\begin{tabular*}{\linewidth}{@{\extracolsep{\fill}}*{10}{c}}                   \hline             
&&&$l_{\infty}$ norm\\
$\tau$& $\alpha=0.4$&  Rate       & $\alpha=0.6$  & Rate       & $\alpha=0.8$ &   Rate    \\\hline
~1/200&   5.2278e-10  &       & 3.4894e-10     &      & 2.2080e-10    &  \\
~1/300&   4.7257e-11  &  5.9279     & 3.1804e-11     & 5.9076     & 2.0415e-11    & 5.8722  \\
~1/400&   8.5294e-12  &  5.9513     & 5.7445e-12     & 5.9488     & 3.7070e-12    & 5.9303  \\
~1/500&   2.2387e-12  &  5.9945     & 1.5159e-12     & 5.9703     & 9.7833e-13    & 5.9698  \\ \hline
&&&discrete $L^{2}$-norm\\
$\tau$& $\alpha=0.4$&  Rate       & $\alpha=0.6$  & Rate       & $\alpha=0.8$ &   Rate    \\\hline
~1/200&   4.7254e-10  &       & 3.0054e-10     &      & 1.7967e-10    &  \\
~1/300&   4.2775e-11  &  5.9245     & 2.7482e-11     & 5.8995     & 1.6721e-11    & 5.8561  \\
~1/400&   7.7276e-12  &  5.9481     & 4.9696e-12     & 5.9447     & 3.0464e-12    &  5.9187  \\
~1/500&   2.0296e-12  &  5.9915     & 1.3156e-12     & 5.9560     & 8.0648e-13    &  5.9560  \\ \hline
    \end{tabular*}\label{table:1}
  \end{center}
\end{table}

From Table \ref{table:1}, we observe the expected convergence rate of BDF6 SAV schemes \eqref{3.5}, \eqref{2.5b}, \eqref{2.5c} and \eqref{2.5d}, which is consistent with the theoretical analysis.

%
%
\section*{Acknowledgments}
The first author wishes to thank Jie Shen  for his valuable comments.

%
%
%
%
%
%

\bibliographystyle{amsplain}

\end{document}